\documentclass{amsart}

\usepackage{frcursive}
\usepackage{amsmath}
\usepackage{amsthm}
\usepackage{amssymb}
\usepackage{hyperref}
\usepackage{color}
\usepackage{mathrsfs}
\usepackage{bbold}

\def\MM{{\mathcal M}}

\def\ZZ{{\mathcal Z}}

\newtheorem{thm}{Theorem}[section]
\newtheorem{lem}[thm]{Lemma}

\newtheorem{cor}[thm]{Corollary}
\newtheorem{defi}[thm]{Definition}
\newtheorem{prop}[thm]{Proposition}

\newcommand{\beqn}{\begin{equation}}
\newcommand{\eeqn}{\end{equation}}
\newcommand{\bear}{\begin{eqnarray}}
\newcommand{\eear}{\end{eqnarray}}
\newcommand{\bean}{\begin{eqnarray*}}
\newcommand{\eean}{\end{eqnarray*}}

\newcommand{\indic}{{\mathbb 1}}

\newcommand{\R}{{\mathbb{R}}}
\newcommand{\T}{{\mathbb{T}}}
\newcommand{\N}{{\mathbb{N}}}

\newcommand{\ep}{\varepsilon}

\newcommand{\vip}{\vskip.2cm}

\newcommand{\E}{\mathbb{E}}

\DeclareMathOperator{\sign}{sign}

\newcommand{\ds}{\displaystyle}

\begin{document}

\title{Mean field limit for the one dimensional  Vlasov-Poisson equation.}

\author{Maxime Hauray}

\address{M. Hauray: LATP, Universit\'e d'Aix-Marseille \& CNRS UMR 7353, 
13453 Marseille Cedex 13, France.}

\email{maxime.hauray@univ-amu.fr}



\thanks{}

\begin{abstract}
We consider systems of $N$ particles in dimension one, driven by pair 
Coulombian or gravitational  interactions. When the number of particles 
goes to infinity in the so called mean field scaling, we formally expect convergence 
towards the Vlasov-Poisson equation.
Actually a rigorous proof of that convergence was given by Trocheris 
in~\cite{Tro86}. Here we shall give a simpler proof of this result,  and explain why it implies the so-called ``Propagation of molecular chaos''. 
More precisely, both results will be a direct consequence of a weak-strong
stability result on the one dimensional Vlasov-Poisson equation that is interesting by it own. We also prove the existence of global solutions to the $N$ particles dynamic starting from any initial positions and velocities, and the existence of global solutions to the Vlasov-Poisson equation starting from any measures with bounded first moment in velocity. 
\end{abstract}

\maketitle


\section{Introduction}
\setcounter{equation}{0}
\label{sec:intro}

\subsection{The $N$ particle system}

We study the evolution of $N$ particles interacting via Coulombian or 
gravitational interaction in dimension one. The 
 position and speed of the $i$-th particle will be denoted respectively by 
$X_i^N$ and $V_i^N$, and we will also use the short-cut $Z_i^N=(X_i^N,V_i^N)$. 
The large vector containing the  positions and speeds of all the particles will 
be denoted by $\ZZ^N = (Z_i^N)_{i \le N}$. 
 
 Our results are valid on the torus and on $\R$, and with (repulsive) Coulombian or (attractive) gravitational interactions.  In fact we will neither use the preservation of the energy and the sign of the potential energy, nor the attractive and repulsive character of the force. In that article, we will consider only the case  of Coulombian interaction in the periodic domain $\T$, in order to keep the notation as simple as possible. The adaptation to the case of the unbounded domain $\R$ does not raise many difficulties, and it is in 
some sense simpler, since for instance the interaction potential and force have simpler expression. The adaptation to the gravitational case is also simple. 

So in all the sequel, the positions $X_i$ belong to $\T$, 
and the velocities $V_i$ to $\R$. On the torus $\T$ that we will identify for simplicity to $\Bigl[ - \frac12,\frac12 \Bigr)$, the Coulombian interaction potential $W$ and the associated force $-W'$ are given   by
 \begin{equation} \label{def:PotW}
  W(x) :=  \frac{x^2 - |x|}2, \qquad 
  - W'(x) := 
 \begin{cases}
  -\frac12 -x  &\text{ if } \; x \in \bigl[-\frac12,0\bigr), \\
 \frac12 -x &\text{ if } \; x \in \bigl(0, \frac12\bigr), \\
 0 & \text{ if } \; x = 0.
\end{cases}
\end{equation}
This particular kernel correspond for instance to the situation where we study electrons in a fixed background of ions.
The equality $W'(0)=0$ may seem strange since the interaction force is singular 
at $0$. But, it will be very 
convenient in the sequel, and will allow for many simplification in the 
notation. That convention corresponds to the fact that there is no self-interaction.
Remark that $W'$ may be decomposed in a singular and a non singular part
\begin{equation} \label{eq:Wdec}
\text{for }\;|x| \le \frac12, \quad 
W'(x) = -x + \frac12 \sign (x),
\qquad \sign(x) := \frac x{|x|}   \quad \bigr(= 0 \text{ if } \; x=0 
\bigr) 
\end{equation}
From these formula, we see that the only singularity of the Poisson kernel $W$ 
is that its second derivative is a Dirac $- \delta_0$. But the interaction force 
$-W'$ is globally bounded, so that in this sense it is not so singular.  The situation is very different in larger 
dimensions, where the Coulombian force is always ``strongly'' singular : $\nabla W \notin 
L^{d/(d-1)}$ in dimension $d$.  


The evolution of the $N$ particles is classically driven by the following system 
of $N$ second order ODEs
\begin{equation} \label{eq:Npart}
\forall \; i \le N, \qquad \dot X_i^N = V_i^N, \qquad  \dot V_i^N = - \frac1N \sum_{j  \neq i} W'(X_i^N-X_j^N).
\end{equation}
Remark that thanks to the assumption that there is 
no self interaction ($W'(0)=0$), we may also use a full sum $\sum_{j=1}^N$ 
instead of the more usual $\sum_{j\neq i}$ in the second equation. The $\frac 1N$ factor is not important for a fixed $N$, but it is necessary if we want to obtain the mean field limit  when $N$ goes to infinity. It appears when  the ``physical'' system is observed at the appropriate scales of time, position and speed. Here, we will not emphasize more on this point, and we consider uniquely non-dimensional equations in the sequel.

\medskip
\subsection{About the existence of solution to the system of ODEs}
As the vector-field that drives this system of ODE is singular (not continous), 
the existence of solutions to that system is not completely obvious.
So to be precise, we shall first define properly the notion of solution that we will 
use. We only define global (in time) solutions, because we will always deal with such solutions in the sequel. 
\begin{defi} \label{def:ODErig}
A global solution of the ODE~\eqref{eq:Npart} with initial condition $\ZZ^N(0)$  is an 
application \mbox{$t  \mapsto \ZZ^N(t)$} defined on $\R^+$
such that  for all $t \in \R^+$, for all $i \le N$
\begin{equation} \label{eq:ODEpre}
X^N_i(t)  = X^N_i(0) + \int_0^t V_i^N(s) \,ds, \qquad 
V_i^N(t) = V_i^N(0) - \frac1N  \sum_{j \neq i} \int_0^t  W'(X_i^N(s)-X_j^N(s)) 
\,ds.
\end{equation}
\end{defi}
An alternative definition, equivalent to that one, would be to require $\ZZ^N$ 
to be differentiable almost everywhere in time  and~\eqref{eq:Npart} to hold 
for all $i \le N$ and a.e. $t \in \R^+$.

In our setting, the now usual theory initiated by DiPerna-Lions 
\cite{DipLions} applies and provides existence and uniqueness of a measure 
preserving flow. Historically, the first result that can be applied here 
was given by Bouchut \cite{Bou02} : it covers the case of second order ODEs of 
the previous type with $BV$ force-field.  But we may also apply the more general 
result of Ambrosio \cite{Amb04}, valid for any ODE with $BV$ vector-field with 
bounded divergence (in $L^\infty$ or even in $L^1$). 

As we shall see later, 
the uniqueness of 
the solution to the system of ODEs~\eqref{eq:Npart} is not important for the mean-field limit, 
and the existence of solutions will be enough for our work.   
The existence of a ``global'' measure preserving flow implies the existence of solutions in the sense of Definition~\ref{def:ODErig} for almost all initial conditions. Remark also  that the uniqueness of the global measure preserving flow does not implies anything for a particular initial condition. But this is not an issue here since we are not interested by uniqueness.

\medskip
But here, in our weakly singular setting,  global existence of solutions to~\eqref{eq:Npart} for \emph{any 
initial condition} may be obtained thanks to the theory of \emph{differential 
inclusion}. Using a result by Filippov \cite{Filippov} we will precisely prove 
the following Proposition
\begin{prop} \label{prop:Nexis}
For any initial configuration $\ZZ^N(0)$, there exists at least one global 
solution to the system of ODE~\eqref{eq:Npart} in the sense of 
Definition~\ref{def:ODErig}.
\end{prop}
The proof of that Proposition is given in Section~\ref{sec:proof}. Thanks to that result, we may now concentrate on the  mean-field limit.

\medskip
\paragraph{ \bf The mean-field limit.}
In the limit of large number of particles, we assume that the initial 
distribution of particles converges towards a (maybe smooth) profile $f_0$. To 
measure this convergence, we introduce the empirical measure 
$\mu_\ZZ^N$ defined by
\begin{equation}
\mu_\ZZ^N(t) :=  \frac1N \sum_{i = 1}^N \delta_{Z_i^N(t)} =  \frac1N \sum_{i = 1}^N \delta_{(X_i^N(t),V_i^N(t))}. 
\end{equation}
We precisely define the convergence towards $f_0$ by the  convergence in the weak sense of measures of $\mu_\ZZ^N$ towards $f_0$
$$
\mu_\ZZ^N(0) \rightharpoonup f_0   \quad \ast-\text{weakly in }\; \MM(\T\times \R).
$$
If the distribution of particles also converges at time $t$ towards a 
distribution profile denoted by $f(t,x,v)$, then we may formally 
replace in the second equation of~\eqref{eq:Npart} the discrete sum by an 
integral of $W'$ against the profile $f(t)$ and get a limit ODE, depending of 
$f$, that drives the limit trajectories of the system
\begin{equation} \label{eq:ODElim}
\qquad \dot X(t) = V(t), \qquad  \dot V(t) = - \int W'(X(t)-y) f(t,y,w) \,dydw.
\end{equation}

From this ODE, using that $f$ must be constant on the trajectories associated 
to~\eqref{eq:ODElim}, we obtain the so-called Vlasov-Poisson equation (that 
should rather be called ``collisionless Boltzmann equation'' or ``Jeans-Vlasov'' 
equation \cite{Henon})
\begin{align} \label{eq:Vla}
&\partial_t f + v \, \partial_x f -  \partial_x  \phi(t,x) \, \partial_v f = 0,\\
&\phi(t,x) = [ W \ast \rho(t) ](x) = \int W(x-y) f(t,y,w) \,dydw 
\nonumber
\end{align}
Here and below we will always refer to $f$ as the distribution of particles (or profile), and to
$$
\rho(t,x) := \int f(t,x,v) \,dxdv
$$
as the density (in position), always denoted by $\rho$, associated to the full distribution $f$.

The argument above is only a formal justification, but a more rigorous one is given in the following Lemma.
\begin{lem} \label{lem:NParttoVla} 
 Assume that $\ZZ^N$ is a global solution to the system of ODE~\eqref{eq:Npart} in the sense of Definition~\ref{def:ODErig}. Then with the assumption that $W'(0)=0$, the associated empirical measure $\mu^N_\ZZ$ is a global weak solution to the Vlasov-Poisson equation~\eqref{eq:Vla}.
\end{lem}
As usual ``weak solution'' means here solution in the sense of distribution.
We will not write the proof of that property which is classical, at least in the  case where $W'$ is Lipschitz. We shall just mention that here, we still have a sufficient regularity to perform the requested calculations. In fact, for any smooth $\varphi$, we may use the chain rule to differentiate in time $t \mapsto \varphi(t,X_i^N(t),V_i^N(t))$ almost everywhere in time. Once the chain rule is applied, we get the result by integration in time and 
summation on all the particles. Remark that the assumption $W'(0)=0$ is crucial here. 

That property is interesting because, it helps to rewrite the problem of the mean field limit in a problem of stability of the Vlasov-Poisson equation~\eqref{eq:Vla}. If the sequence of initial empirical measures $\bigl(\mu^N_\ZZ(0)\bigr)$ converges towards some distribution $f^0$, the mean filed limit will holds if :
\begin{itemize}
  \item[i)] there is a unique solution starting form $f^0$ in a suitable class,
  \item[ii)] that solution is stable (in time)  when we approach it by sum of 
Dirac masses.
\end{itemize}

In that article, we shall in fact show a "weak-strong" stability principle for the Vlasov-Poisson equation~\eqref{eq:Vla}: if 
a solution $f$ of the Vlasov equation has a  
density $\rho$ that remains bounded (uniformly in position) along time, then it is stable  in the class of 
measure solutions with finite first order moment in $v$. Of course, this stability result implies the uniqueness in the class of measures (with finite first moment in $v$)
of the solution starting from the initial condition $f^0$ of $f$. For that reason, that kind of result is more commonly 
called  ``weak-strong'' uniqueness principle. But we choose to emphasize on 
the stability, that is maybe more important for the mean-field limit.

\medskip
\subsection{Mean-field limit using notation from probability.}

The system of ODE~\eqref{eq:ODElim} has also an interesting reformulation if we use an idea and a notation borrowed to the theory of probability and more precisely to the theory of continuous stochastic processes. We may say that the  $X_t =X(t)$ and $V_t=V(t)$ appearing in the ODE~\eqref{eq:ODElim} are trajectories of a 
``non stochastic'' process. Even if their is no noise, the probabilistic notation is quite useful for two reasons :
\begin{itemize}
 \item We can interpret $f^0$ as the law of the initial position and speed $(X_0,V_0)$,
 \item the low regularity of the interaction force $W'$ implies that the 
solutions of~\eqref{eq:ODElim} are not unique if $f_t$ is 
not regular enough. This introduces naturally some ``randomness'', as some choices have to be made for the 
construction of solutions. 
\end{itemize}

We mention this ``probabilistic'' point of view  was already introduced for 
(deterministic) ODE by Ambrosio in his study of linear 
transport and associated ODE with low regularity \cite[Section 5]{Amb04} (See also references therein for earlier works with a similar point of view). In fact, we shall use several times a results proved in its  well written lecture notes \cite{AmbNotes}). 
His ideas were also generalized later by Figalli to the case of SDEs  with possibly degenerate noise \cite{Fig08}. 

Here we mostly adopt this point of view in the nonlinear setting, and with a slightly different notation.
Precisely, the system~\eqref{eq:ODElim} maybe 
rewritten as a \emph{nonlinear ODE}. $f_t = f(t)$ will be the law at time $t$ 
of the couple $Z_t=(X_t,V_t)$ and  we use the notation $Z_\cdot$ to refer to a 
trajectory, \emph{i.e.} an element of $C([0,+\infty),\T \times \R)$.

\begin{defi}[Nonlinear ODE] \label{def:ODENL}
A solution of the non linear ODE below with initial condition $\nu^0 \in \mathcal P(\T \times \R)$ is a probability $\mathbb P$ on $C([0,+\infty),\T \times \R)$ such that  $\mathbb P(Z_0 \in \cdot) = \nu^0$ and
the following equalities hold  for $\mathbb P \times \mathbb P$-almost every $(Z_\cdot ,\bar Z_\cdot)$
\begin{equation} \label{eq:ODENL}
\forall  t \in [0,+\infty), \quad \begin{cases} 
\ds  X_t = X_0 + \int_0^t V_s \,ds, \\
 \ds V_t = V_0  - \int_0^t \E_{\bar Z_\cdot} \bigl[ W'(X_s- \bar X_s)  \bigr] \,ds.
\end{cases}
 \end{equation}
If as usual, we prefer to speak with random variables than with laws, we must 
just say that $Z_\cdot$ is a random variable, that the law of $Z_0$ is $\nu_0$, 
and that the previous equation are satisfied with $\bar Z_\cdot = (\bar 
X_\cdot,\bar V_\cdot)$ an independent copy of $Z_\cdot$, \emph{i.e.} a 
independent process with same law than $Z_\cdot$.  Remark that as in the 
previous system of ODEs, the requirement of equation~\eqref{eq:ODENL} is 
equivalent to ask that the associated equality on $\dot X_t$ and $\dot V_t$ are 
true almost surely and for almost all time $t \in [0,+\infty)$.
\end{defi}

 In fact this definition is just a reformulation of the 
system~\eqref{eq:ODElim}, with expectation instead of integral, and with the 
vocabulary of stochastic processes. 
 A process $Z_\cdot$ solution to~\eqref{eq:ODENL} is concentrated on deterministic solutions to~\eqref{eq:ODElim}. But it has some interest when we 
will use optimal transport and coupling techniques. The probabilistic notation in 
which the  effective probability spaces and optimal mappings are never 
explicitly written is sometimes more simple to handle. 
 It also allows to work directly on trajectories with a
 simple notation.
 
 We refer to~\eqref{eq:ODENL} as a nonlinear ODE, in reference to the nonlinear SDEs introduced by 
McKean in~\cite{McKean}, also known as McKean-Vlasov processes. These nonlinear 
SDEs correspond to add a brownian drift (or more generally a random force field) 
in the second equation of~\eqref{eq:ODENL}, and possibly the action of an 
exterior potential. As in our deterministic case, these nonlinear SDEs are 
natural limits of N particle systems with noise \cite{McKean2,Meleard}.
  


In the case of smooth interactions, the so-called \emph{methods of characteristics} implies that the Vlasov-Poisson equation~\eqref{eq:Vla} and associated the nonlinear ODE~\eqref{def:ODENL} give exactly the same solutions. In fact, when the force field is sufficiently regular for uniqueness of the trajectories to hold, the two problems are equivalent.  With non smooth interactions, the uniqueness of trajectories is lost and we are only able to state the following lemma. 
 \begin{lem} \label{lem:ODENL_mart}
 Assume that $Z_\cdot$ is a solution of the nonlinear ODE~\eqref{eq:ODENL}. 
Then, the law $\nu_t$ of its time marginals $Z_t$ is a weak solution of the 
Vlasov-Poisson equation~\eqref{eq:Vla}.

 Reciprocally, if $\nu_t$ is a bounded measure solution to 
the Vlasov-Poisson equation~\eqref{eq:Vla}, there 
exists a process $Z_\cdot$ solution to~\eqref{eq:ODENL} whose time marginals 
are exactly the $\nu_t$.  
 \end{lem}

Without noise, the first part of this lemma is just an application of the chain rule. 
The second part of the lemma is more delicate. It is a consequence of a theorem 
of Ambrosio, for linear transport equation. In fact $\nu_t$ is a solution of 
the linear transport equation
$$
\partial_t \mu + v \partial_x \mu + E[\nu_t] \partial_v \mu = 0,
$$
where the fixed force field is given by $E[\nu_t] := - \int 
W'(x-y) \nu_t(dy,dw)$.  In that setting, we can apply 
\cite[Theorem 3.2]{AmbNotes}
to obtain a process $Z_\cdot$ solution to~\eqref{eq:ODENL} with time marginals
$\nu_t$.

\medskip
Another 
interesting property associated to these nonlinear ODE is the following
\begin{lem} \label{lem:NPtoODENL}
If $\ZZ^N_\cdot$ is a solution to the N particles system~\eqref{eq:Npart}, then the empirical measure on the trajectories
$$
\mathbb P^N_{\ZZ_\cdot} := \frac1N \sum_{i=1}^N \delta_{Z^N_{i,\cdot}}
$$
is a solution of the nonlinear ODE~\eqref{eq:ODENL} with initial condition $\mu_\ZZ^N(0)$.
\end{lem}

As for Lemma~\ref{lem:NParttoVla}, this is more a way of rewriting the equations~\eqref{eq:Npart}, 
and we will not be more precise. Note that there is no reciprocal statement : 
along solutions of the nonlinear ODE~\eqref{eq:ODENL} starting from sum of Dirac masses, 
the singularity of the Force allows to split some Dirac masses in smaller ones and also in less singular measures supported on a line \cite{ZM2,ZM3,DFV13}.

\medskip
Using approximation by sum of Dirac masses (a strategy already used in 
aggregation equation, see \cite{CDFLS} and references therein), we will be able 
to give another proof of a general existence theorem for the one-dimensional VP equation obtained by Zheng and Majda~\cite{ZM1}
\begin{thm}\label{thm:ODENLex}
Assume that $\nu^0$ is a probability measure with finite first moment in velocity : $\int |v| 
\,\nu^0(dx,dv) < \infty$, and recall that we use the convention $W'(0)=0$.
Then there exists (at least one) global solution to the Vlasov equation~\eqref{eq:Vla}, 
with initial condition $\nu^0$, and also one process solution of the nonlinear 
ODE~\eqref{eq:ODENL} with the same initial condition. 
\end{thm}

The proof of Theorem~\ref{thm:ODENLex} is done in Section~\ref{sec:proof}. Here we will only make some comments : 

\begin{itemize}
\item In order to define properly the interaction force (which has to be defined everywhere if we want to define measure solutions) we have introduced the convention $W'(0)=0$. Madja and Zheng~\cite{ZM1} use a strategy that may seem different, but both are in fact equivalent.  

\item Our result is more general that the one of Zheng and Majda, which requires a exponential moment (in $v$) on the initial measure $\nu^0$.

\item Our proof is also in some sense more simple. It relies on the approximation of the initial data by sum of Dirac masses, for which we known by Proposition~\ref{prop:Nexis} that there exists global solutions to~\eqref{eq:Npart} and thus~\eqref{eq:ODENL}. Then, the tightness of the law of these processes is 
obtained rather easily, but more work is required to characterize the limit. For the last point, we first proof that the time marginals associated to any limit process are solution of the VP equation~\eqref{eq:Vla}. Once the time marginals, and therefore the force field are known, we can 
apply the second point of Lemma~\ref{lem:ODENL_mart} to get the existence of a process solution to~\eqref{eq:ODENL} with the requested time marginals. But we were not able to prove that the limit process is a solution to the nonlinear ODE~\eqref{eq:ODENL}. Of course, the use of the nonlinear ODE~\eqref{eq:ODENL} is not mandatory 
if we are interest only by the result on the Vlasov-Poisson equation, but it will nevertheless be convenient, since we will need to perform some estimates on the trajectories.
\end{itemize}

%
%
%

\subsection{A weak-strong stability estimate around solutions to VP equation with bounded density}

Here we will state the main result of our article :  a weak-strong stability 
estimate around solutions of the Vlasov-Poisson equation~\eqref{eq:Vla} with 
bounded density. Recall that in this article we use the word ``density'' only 
to refer to the density in physical space $\rho(t,x) = \int f(t,x,v) \,dv$.

We introduce the Monge-Kantorovitch-Wasserstein distance of order one, denoted 
by $W_1$ constructed with the Euclidian distance on $\T\times \R$. We refer to the 
clear book of Villani \cite{Vill03} for an introduction to this object. We will 
use it both with random variables or their law (which are probabilities) in the argument, depending of 
the situation.

\begin{thm} \label{thm:proc}
Assume that $Z^1_\cdot$ is a solution of the nonlinear ODE~\eqref{eq:ODENL}, 
with time marginals $f_t$ that have bounded density $\rho_t$ for any time $t 
\ge 0$. 
Then for any global solution $Z^2_\cdot$ of the nonlinear ODE~\eqref{eq:ODENL} with finite first order moment in $v$, 
we have the following stability estimate
$$
\forall t \in \R^+, 
\qquad
W_1( Z^1_t, Z^2_t) \le  e^{a(t) } W_1(Z^1_0,Z^2_0),
\qquad \text{with}
\quad 
a(t) := \sqrt 2 \, t + 8 \int_0^t \| \rho_s \|_\infty  \,ds.
$$
\end{thm}

\begin{thm} \label{thm:VP}
Assume that $f_t$ is a solution of the Vlasov Poisson 
equation~\eqref{eq:Vla} 
with bounded density $\rho_t$ for any time  $t \ge 0$. Then for any 
global measure solution $\nu$
to the same equation with finite first order moment in $v$, we have the following stability estimate
$$
\forall t \in \R^+, 
\qquad
W_1( f_t, \nu_t) \le  e^{a(t)} W_1(f^0,\nu_0),
\qquad \text{with}
\quad 
a(t) := \sqrt 2 \, t + 8 \int_0^t \| \rho_s \|_\infty  \,ds.
$$
\end{thm}

Theorem~\ref{thm:VP} implies Theorem~\ref{thm:proc}, since time marginals of solutions of the nonlinear ODE~\eqref{eq:ODENL} are solutions of the 
Vlasov-Poisson equation~\eqref{eq:Vla}. But, the converse is also true, since  
thanks to the second part of Lemma~\ref{lem:ODENL_mart}, any solution of the 
VP equation may be represented as time marginals of a solution of the 
nonlinear ODE. In view of that, we will only proof  Theorem~\ref{thm:proc} in 
section~\ref{sec:proof}. But remark also that this proof written with the 
formalism of processes maybe adapted to the formalism of solution of Vlasov -Poisson
equation~\eqref{eq:Vla}, with less simple notation (in our opinion). So, again, the choice of working with VP equation of the nonlinear ODE is rather a matter of convenience.

\subsection{Existence of strong solutions.}
The later result is interesting only if solutions with bounded density do 
exist. But, this is a well-known fact \cite{CK80,Bos05,Guo95,LMR10}. 
Here we will restate a Proposition about the existence of such solutions that 
relies precisely on Theorem~\ref{thm:proc} or~\ref{thm:VP}. Remark that theorem~\ref{thm:VP} implies that such solutions are automatically unique.

\begin{prop} \label{prop:exisVP}
Assume that $f_0$ satisfies 
\begin{equation} \label{eq:condL1}
 \int_< 0^{+\infty} g_0(v) \,dv < + \infty, \quad
 \| f_0 \|_\infty = g(0) < + \infty,
 \qquad \text{where} \quad
 g_0(v) := \sup_{y \in \T, |w| \ge v} |f_0(x,w)|.
\end{equation}
Then there exists one unique solution $Z_\cdot$ to the nonlinear 
ODE~\eqref{eq:ODENL}, with initial condition $f^0$.
There also exists a unique solution to the Vlasov-Poisson 
equation~\eqref{eq:Vla} with initial condition $f^0$, given by the time 
marginals $f_t$ of $Z_\cdot$.

The density (in position) $\rho_t$ associated to $f_t$ satisfies in particular 
the bound
$$
 \| \rho_t \|_\infty \le  2 \int_0^{+\infty} g_0(v) \,dv + 
\|f_0\|_\infty t.
$$
\end{prop}

\subsection{Mean-field limit.}
If the initial condition $f_0$ satisfies the 
condition~\eqref{eq:condL1}, the mean-field limit around the unique solution $f$ starting from $f_0$  is a direct consequence of 
Theorem~\ref{thm:VP}. In fact, it suffice to apply this theorem to the particular case 
of initial data $\nu_0$ given by sum of $N$ Dirac masses. This is precisely 
stated in the following corollary, for which we will not give a specific proof.
\begin{cor}[of Theorem~\ref{thm:VP}] \label{cor:MFL}
Assume that $f_0$ satisfies the condition~\eqref{eq:condL1} and denote by 
$f_t$ the unique solution of the VP equation~\eqref{eq:Vla} with initial condition $f_0$.
Proposition~\ref{prop:exisVP} ensures that its  density $\rho_t$ is bounded.
For any $\ZZ^N$ solution of the $N$ particles system~\eqref{eq:Npart} with 
initial empirical measure $\mu^N_\ZZ(0)$, we have the following stability
estimate
\begin{equation} \label{eq:MFL}
W_1(\mu^N_\ZZ(t),f_t) \le e^{a(t)} W_1(\mu^N_\ZZ(0), f_0)
\qquad \text{with}
\quad 
a(t) := \sqrt 2 \, t + 8 \int_0^t \| \rho_s \|_\infty  \,ds.
\end{equation}

\end{cor}

This implies the mean-field limit in large number of particles since if the 
positions and velocities of the particles are chosen so that 
$$
\mu^N_\ZZ(0) \xrightarrow
[N \rightarrow +\infty]{} f^0  \; \text{ weakly},
\qquad
\sup_{n \in \N} \int |v| \mu_\ZZ^N(0,dxdv) 
\rightarrow
\int |v| f_0(dx,dv),
$$
then $W_1(\mu_\ZZ^N(0),f_0) \rightarrow 0$ (See \cite{Vill03}). The above corollary implies that the weak convergence holds at any time. 
Remark that we may also obtain a similar result on the trajectories, \emph{i.e.} 
on the solutions of the nonlinear ODE~\eqref{eq:ODENL}.

\subsection{Propagation of molecular chaos}

Around profiles with bounded density, the propagation of chaos is also a 
consequence of the ``weak-strong'' stability theorem. 
\begin{cor}[of Theorem~\ref{thm:VP}]
Assume that $f^0$ satisfies the condition~\eqref{eq:condL1}. Then 
with the same notation than in Corollary~\ref{cor:MFL}, we have
$$
\E \bigl[ 
W_1(\mu^N_\ZZ(t),f_t)   \bigr] 
\le e^{a(t)} \E \bigl[ W_1(\mu^N_\ZZ(0), f_0) \bigr]
$$
\end{cor}

We shall not give the proof of that results, which is obtain by taking the expectation in~\eqref{eq:MFL}. We just make some comments :
\begin{itemize}
\item it is also possible to obtain other related results, for instance bounds 
on the 
exponential moments of $W_1(\mu^N_\ZZ(t),f_t)$ can be obtain in terms of 
exponential moments of $W_1(\mu^N_\ZZ(0), f_0)$,
which are known to exists, see for instance \cite{BoissardPhD},

\item the result may be adapted to the ``language'' of nonlinear ODE,

\item our results are stated in terms of convergence in law of the empirical 
measure towards $f_t$. For the relations to the more usual convergence of 
$k$-particles marginals ($k\ge 2$), we refer for instance to the famous lecture notes by 
Sznitman~\cite{Sznitman}. We also mentioned that 
quantitative equivalence estimates were also recently obtained 
in \cite{MischMou} and \cite{HauMisch}.
\end{itemize}

%

\section{Proofs of the main results} \label{sec:proof}
%
%

\subsection{Proof of Proposition~\ref{prop:Nexis}}

The existence will be proved with the help of the theory of differential inclusion \cite{Filippov}. 
First we will construct an associated differential inclusion, to which solutions do exist, and then prove that these solutions are indeed solutions to the original problem.

\medskip
{\sl Step 1. The construction of an adapted differential inclusion.}

We will replace the system of equation~\eqref{eq:Npart} by a differential inclusion
\begin{equation} \label{eq:inc}
\dot \ZZ^N(t) \in \mathcal B^N \bigl(\ZZ^N(t)\bigr),
\end{equation}
where $ \mathcal B^N$ is set-valued, \emph{i.e.} is an application from $\R^{2N}$ into $\mathcal P(\R^{2N})$. A Theorem by Filippov \cite[Chapter 2, Theorem 1]{Filippov} ensures the existence of global solutions to the differential inclusion~\eqref{eq:inc} for any initial condition if :
\begin{itemize}
 \item for each $\ZZ^N \in \R^{2N}$, the set $\mathcal B^N \bigl(\ZZ^N\bigr)$ is bounded, and  closed, and  convex,
 \item $\mathcal B^N$ is locally bounded, \emph{i.e.}  for any compact set $K \in \R^{2N}$ there exists a compact set $K' \in \R^{2N}$ such that $\mathcal B^N(\ZZ^N) \subset K'$ for all $\ZZ^N \in K$, 
 \item $\mathcal B^N$ is upper continuous with respect to the inclusion. 
\end{itemize}

A natural $\mathcal B^N$ may be constructed as follow. We will not write the full application $\mathcal B^N$, but assume that it has the form 
$$
\dot \ZZ^N(t) \in \mathcal B^N \bigl(\ZZ^N(t)\bigr)
\;\Leftrightarrow \;
\dot X^N_i = V_i, \qquad \dot V_i = \frac1N \sum_i F_{ij}^N,
$$
where the forces $F^N_{ij}$ satisfy some conditions to be specified.

A first choice for the $F^N_{ij}$ is to ask that for all $i,j \le N$:
\begin{itemize}
\item $\ds F_{ij}^N = -W'(X^N_i-X^N_j)$ when $X^N_i \neq X^N_j$,
\item $\ds F_{ij}^N \in \Bigl[-\frac12, \frac12 \Bigr]$ when $X^N_i = X^N_j$.
\end{itemize}
The associated $\mathcal B^N$ satisfies all the required properties, 
but unfortunately, the associated solutions will not be solution of the original ODE system in general.
For instance if you take two particles, both originally  at position $0$, with velocity $0$, then according to~\eqref{eq:Npart} and the fact that $W'(0)=0$, they should never move. But with our construction of $\mathcal B^N$, the solution where the two particles remain stuck together and accelerate uniformly with acceleration $\frac12$ is admissible. 

To overcome this problem, we should define $\mathcal B^N$ almost as before, but with the additional conditions that we want the \emph{action-reaction principle} to be valid in any case. The precise conditions are that for all $i <j \le N$:
\begin{itemize}
\item $\ds F_{ij}^N = -F_{ji}^N = -W'(X^N_i-X^N_j)$ when $X^N_i \neq X^N_j$,
\item $\ds F_{ij}^N = -F_{ji}^N \in \Bigl[-\frac12, \frac12 \Bigr]$ when $X^N_i = X^N_j$.
\end{itemize}
It is not difficult that to see that this $\mathcal B^N$ still fulfils the requested properties, and \cite[Chapter 2, Theorem 1]{Filippov} ensures the existence of a solution also denoted $\ZZ^N$ to the associated differential inclusion.   

\medskip
{\sl Step 2. Solutions of the differential inclusion are solutions to the $N$ particles problem. }

According to the final definition of $\mathcal B^N$, we see that problems may occur only when two particles are at the same position. Here, we will only deal with particles $1$ and $2$, the others case being similar. But since for any $t \in \R^+$ and any $\lambda >0$, the set $\bigl\{ s \in [0,t] \; \text {s.t. }  X^N_1(s) =X^N_2(s)
\; \text{ and }   | V^N_1(s) - V^N_1(s) | \ge \lambda \bigr\}$, is made of isolated points, we have also that 
the set
$$
A_t := \bigl\{ s \in [0,t] \; \text {s.t. }  X^N_1(s) =X^N_2(s)
\; \text{ and }   V^N_1(s) \neq  V^N_2(s)  \bigr\}
$$
is negligible with respect to the Lebesgue measure. 
Since the exact value of the force used on negligible set  has no influence on the solution,  it remains only to understand what happens on the set
$$
 B_t := \bigl\{ s \in [0,t] \; \text {s.t. }  X^N_1(s) =X^N_2(s) \; \text{ and }   V^N_1(s) =  V^N_1(s)  \bigr\}
$$
For this set, we will use the following lemma, that we shall prove later.
%
\begin{lem} \label{lem:Winf}
Assume that $g : [0,t] \rightarrow \R$ is a Lipschitz function (differentiable  almost everywhere by the Rademacher theorem). Then the set 
$$
\{   s \in [0,t] \; \text{ s.t. } \; g(0) = 0 \;\text{ and } \; g'(0) \neq 0   \} \quad \text{is negligible.}
$$
\end{lem} 

Applying this Lemma with the function  $g(s) = V^N_1(s)-  V^N_2(s)$ we get that the set
$$
\bigl\{ s \in [0,t] \; \text {s.t. }  X^N_1(s) =X^N_2(s) \; \text{ and } \;  V^N_1(s) = V^N_1(s)  
 \; \text{ and } \; 
 (V^N_1)'(s) \neq   (V^N_1)'(s)
 \bigr\}
$$
is negligible. Then the only part we have really to deal with is
$$
C_t := \bigl\{ s \in [0,t] \; \text {s.t. }  X^N_1(s) =X^N_2(s)  \; \text{ and } \;  V^N_1(s) = V^N_1(s)  
 \; \text{ and } \; 
 (V^N_1)'(s) =  (V^N_1)'(s)
 \bigr\}.
$$

But remark that for $ s \in C_t$ the particles are at the same position and are submitted to the  same  global force. If no other particle share their common position at this time, it follows from the (final) definition of $\mathcal B^N$ that $F_{12}^N = -F{21}^N=0$.
If there is another particle that share their common position, the situation is a little bit different.
Say for instance that the third particles is also at the same position. Then, the configuration where $F_{i,i+1} = -F_{i+1,i} = \frac12$ (whit the convention $3+1 =1$) is admissible. In that case , $F_{12} \neq 0$, but it is also clear that we can set all the forces $F_{i,i+1}$ to zero without changing anything to all the global forces seen by the three particles. 

In general, since the particle $1$ and $2$ are submitted to  the same force at time $s$ (remember that $s \in C_t$), and because the action-reaction principle is preserved  by the differential inclusion, this common force is also the one that applies to their center of mass. And in the case where $F_{12}(s) \neq 0$, we will not change the solution if we replace the $F_{ij}^N$ by the force $\tilde F_{ij}^N$ acting on their centrer of mass, defined by 
$$
\tilde F_{12}=0 , \qquad \tilde F_{1i}= \tilde F_{2i} = - \tilde F_{i1} = - \tilde F_{i2}= \frac12 (F_{1i}+ F_{2i}) \; \text{ for }\; i \ge 3.
$$  
In fact that replacement is acceptable by the convexity of $\mathcal B^N$, and this modification does not affect the global force seen by each particle. It means that the original equation\eqref{eq:Npart} is already fulfilled at this time $s$.

Finally, we have proved that the system of equations~\eqref{eq:Npart} is satisfied for almost all time and for all couple $i \neq j$. As already said, this is equivalent to the definition~\ref{def:ODErig}.

\begin{proof}[Proof of Lemma~\ref{lem:Winf}]
For simplicity we will assume that $g$ has a Lipschitz constant  equals to one.  If for all $\lambda>0$, the set $A_ \lambda := \{   s \in [0,t] \; \text{ s.t. } \; g(0) = 0 \;\text{ and } \; g'(0)  \ge  \lambda  \}$ is negligible, and if also a similar statement is true for the negative value of $g'$, then it is not difficult to conclude. 

So we will prove by contradiction that $A_\lambda$ is of zero measure for any $\lambda >0$. If not, then we may choose a $\lambda >0$ such that 
 $A_\lambda$ has strictly positive measure. It is then classical that $A_\lambda$ possesses an accumulation point  $s$ (See for instance \cite{Stein}, in particular for the argument with non centered interval), for which
$$
\lim_{\ep,\ep' \rightarrow 0} \alpha(\ep,\ep') = 1, \qquad
\text{where} \quad \alpha(\ep,\ep') := 
\frac{  \bigl| 
\{   u \in [s-\ep,s+ \ep'] \; \text{ s.t. } \; g(0) = 0 \;\text{ and } \; g'(0)  \ge  \lambda   \}
 \bigr|   }{ \ep+\ep'} 
 $$ 
But we also have
\begin{align*}
g(s+\ep') - g(s-\ep) & = \int_{s-\ep}^{s+\ep'} g'(r) \,dr \ge 
(\ep + \ep') \bigl[   \lambda \alpha(\ep,\ep') - (1-\alpha(\ep,\ep')) \bigr]
\\
 & \ge 
(\ep + \ep') \bigl[  (1+ \lambda) \alpha(\ep,\ep') - 1 \bigr],\\
& \ge \frac \lambda 2 (\ep + \ep') \qquad \text{for all } \; \ep,\ep' \; \text{ small enough}. 
\end{align*}
But this contradicts the fact that $s$ is an accumulation point for $A_\lambda$. 
\end{proof}

\subsection{Proof of Theorem~\ref{thm:ODENLex}}
 
 The proof is separated in five steps. In the first one, we construct an approximating sequence made of solutions to the N particles systems. In the second, we show the tightness of that sequence, and in the last three ones, we show that the time marginals of the limit process are solution of the Vlasov-Poisson equation~\eqref{eq:Vla}. Afterthat, the existence of a process solution to the nonlinear ODE~\eqref{eq:ODENL} can be obtained by an application of the second point of Lemma~\ref{lem:ODENL_mart}.
 
 \medskip
{\sl Step 1. Construction of the approximating sequence.}

This is quite simple. For each N choose a sequence $\ZZ^N_0$ of initial positions and velocities, such that the associated empirical measures $\mu^N_0 = \mu_\ZZ^N(0)$ converges weakly towards the measure $\nu_0$.  Since $\int |v| \nu_0(dx,dv) <+ \infty$, we may also assume that
\begin{equation} \label{estim:mominit}
\sup_{N \in \N} \int |v|  \mu_\ZZ^N(0,dx,dv) < + \infty.
\end{equation}
Applying Proposition~\ref{prop:Nexis}, we obtain a global solution $\ZZ^N_\cdot$ to the system of ODE~\eqref{eq:Npart}.  Then Lemma~\ref{lem:NPtoODENL} ensure that it also define a process denoted by $Z^N_\cdot$, solution to the nonlinear ODE~\eqref{eq:ODENL} with initial 
condition $\mu_\ZZ^N(0)$.

\medskip
{\sl Step 2. Tightness of the approximating sequence.}
Using the equation~\eqref{eq:ODEpre}, we get that for almost all $(s,u)$ satifying $0<s<u<t$ 
\begin{equation} \label{estim:VLip}
 |V^N_u -V^N_s|   \le  \int_s^u \E_{\bar Z_\cdot }\bigl[ | W'(X^N_r - \bar X^N_r| \bigr] \,dr  
 \le \frac12 |u-s| .
 \end{equation}
And with the particular choice  $s=0$, we obtain  the bound $| V_u^N|  \le  |V_0^N| + \frac u 2$.
That bound allows to obtain a similar estimate for the position
\begin{equation} \label{estim:XLip}
|X^N_u -X^N_s|   \le \int_s^u     |V^N_r| \,dr \le \Bigl( |V^N_0| + \frac t2 \Bigr) |s-u|.
\end{equation}
Merging estimates~\eqref{estim:VLip} and~\eqref{estim:XLip} and taking the supremum on all couple $(u,s)$ we get
$$
\sup_{s,u \in [0,t]} \frac{ |Z^N_u -Z^N_s|}{|u-s|} \le  |V^N_0| + \frac {t+1}2.
$$
Taking the expectation and using~\eqref{estim:mominit}, we get that
\begin{equation} \label{estim:exp}
\sup_{N \in \N}   \E \bigg[  \sup_{s,u \in [0,t]} \frac{ |Z^N_u -Z^N_s|}{|u-s|} \biggr]  < + \infty.
\end{equation}
This implies the tightness. In fact, by the Arzel\`a-Ascoli theorem, the subsets  $\mathcal B_R^\lambda(t)$ of $C([0,t],\T \times \R)$ defined by 
$$
\mathcal B_R^\lambda(t) := \Bigl\{  Z_\cdot \text{ s.t. }   |X_0| \le R \quad \text{and} \quad 
 \bigl\| Z' \bigr\|_\infty \le  \lambda 
 \Bigr\}
$$
are compact for all $R,\lambda,t \in \R^+$. But the fact that 
$\mathbb P(Z_0^N \in \cdot) = \mu^N_0$, together with estimates\eqref{estim:mominit} and~\eqref{estim:exp} imply that the compact sets $\mathcal B_R^\lambda(t)$ are almost of full measure when $R$ and $\lambda$ are large.

Therefore, up to the extraction of a subsequence, we may assume that the sequence of processes $(Z^N_\cdot)$ converges weakly towards a process $Z_\cdot$. To understand this properly, we must say that $C([0,+\infty),\T \times \R)$ is endowed with the topology of uniform convergence on every bounded time interval.

Remark also that the processes $Z^N_\cdot$ are concentrated on the set of trajectories such that $t \mapsto V^N_t$ are uniformly Lipschitz with constant $\frac12$, a set that is closed under the topology of uniform convergence on any bounded time interval. So this is also true for the process $Z_\cdot$ : almost surely, the velocity component of the limit process $Z_\cdot$ has Lipschitz trajectories, with constant $\frac12$.

\medskip
{\sl Step 3. Characterization of the limit process $Z_\cdot$.}

First by construction it is clear that the limit process $Z_\cdot$ satisfies the required initial condition.
Next we will show that its time marginals satisfy the Vlasov-Poisson equation~\eqref{eq:Vla}.
For this choose two times $s < t$ and a smooth function $\varphi$. We need to show that
\begin{equation} \label{eq:mart2}
\begin{split}
\E_{(Z_\cdot,\bar Z_\cdot)} \biggl[  \Bigl(
\varphi(t,Z_t) - \varphi(s,Z_s)  - \int_s \bigl[ \partial_t \varphi(u,Z_u) + V_s \partial_x \varphi(u,Z_u) 
\\
- W'(X_u -\bar X_u) \partial_v \varphi(u,Z_u) \bigr] \,du
\Bigr) \biggr] = 0.
\end{split}
\end{equation}
since $Z^N$ are solution of the nonlinear ODE~\eqref{eq:ODENL}, it is clear that this is satisfied if we replace $Z_\cdot$ by $Z_\cdot^N$. Because of the continuity of $\varphi$, we can use classical results, see  Billingsley \cite[p. 30]{Bill},  and  pass to the limit in 
$$
\E_{(Z^N_\cdot,\bar Z^N_\cdot)} \bigl[  \varphi(t,Z^N_t) - \varphi(s,Z^N_s) \bigr]
\xrightarrow[\;N \rightarrow + \infty \;]{}
\E_{(Z_\cdot,\bar Z_\cdot)} \bigl[  \varphi(t,Z_t) - \varphi(s,Z_s) \bigr].
$$
In fact, only the term containing $W'$  is more involved, since the interaction force is discontinuous at the origin.
To control it, we need to separate the distant and close interactions.
For this, we choose an $\ep>0$ and introduce a smooth cut-off function $\chi_\ep$ defined on $\T$ and satisfying (we use again the identification $\T = \bigl[-\frac12, \frac12 \bigr)$)
$$
\chi_\ep(x) = 1 \;\text{ if } \; |x|  \le \ep, \qquad \chi_\ep(x) = 0 \;\text{ if  }\; |x| \ge  2 \ep.
$$
Then, we can decompose 
\begin{align*} 
\E_{(Z^N_\cdot,\bar Z^N_\cdot)} \biggl[  \int_s^t  & W'(X^N_u -\bar X^N_u)  \partial_v \varphi(u,Z^N_u) \bigr] \,du  \biggr] \\
& = \E_{(Z^N_\cdot,\bar Z^N_\cdot)} \biggl[  \int_s^t W'(X^N_u -\bar X^N_u) (1 - \chi_\ep(X^N_u-\bar X^N_u))  \partial_v \varphi(u,Z^N_u) \bigr] \,du  \biggr] \\
& \hspace{10mm} +
\E_{(Z^N_\cdot,\bar Z^N_\cdot)} \biggl[  \int_s^t W'(X^N_u -\bar X^N_u)  \chi_\ep(X^N_u-\bar X^N_u) \partial_v \varphi(u,Z^N_u) \bigr] \,du  \biggr].
\end{align*}
In the first term of the r.h.s that  take into account the interactions between distant particles, everything is now continuous. So, as before, it is not difficult to pass in the limit in that term. It converges towards 
$$
\E_{(Z_\cdot,\bar Z_\cdot)} \biggl[  \int_s^t W'(X_u -\bar X_u) \bigl(1- \chi_\ep(X_u-\bar X_u) \bigr) \partial_v \varphi(u,Z_u) \bigr] \,du  \biggr]  
$$
It remains to control the second term. We will show that it goes to zero as $\ep$ goes to zero, uniformly in $N$, and that this convergence also holds for the limit process $Z_\cdot$. This will conclude the proof.

\medskip
For this we choose a $\beta < \frac12$ and introduce a second smooth cut-off function defined on $\R$, such that
$$
\xi_\ep(v) = 1 \;\text{ if } \; |v|  \le \ep^\beta, \qquad \xi_\ep(v) = 0 \;\text{ if  }\; |v| \ge  2 \ep^\beta.
$$
We use it to separate the remaining term in two : one term taking into account the close interactions with large relative speed, and the second one taking into account the close interactions with small relative speed
\begin{align*} 
\E_{(Z^N_\cdot,\bar Z^N_\cdot)} \biggl[  & \int_s^t W'(X^N_u -\bar X^N_u)  \chi_\ep(X^N_u-\bar X^N_u) \partial_v \varphi(u,Z^N_u) \bigr] \,du  \biggr] \\
& = \E_{(Z^N_\cdot,\bar Z^N_\cdot)} \biggl[  \int_s^t W'(X^N_u -\bar X^N_u) \chi_\ep(X^N_u-\bar X^N_u) 
(1- \xi_\ep(V^N_u-\bar V^N_u) )  \partial_v \varphi(u,Z^N_u) \bigr] \,du  \biggr] \\
& \hspace{10mm} +
\E_{(Z^N_\cdot,\bar Z^N_\cdot)} \biggl[  \int_s^t W'(X^N_u -\bar X^N_u)  \chi_\ep(X^N_u-\bar X^N_u) 
\xi_\ep(V^N_u-\bar V^N_u) \partial_v \varphi(u,Z^N_u) \bigr] \,du  \biggr].
\end{align*}

\medskip
{\sl Step 4. The effects of close interaction with large relative speed.}
To understand the first term in the r.h.s. of the above decomposition, we choose a particular couple of trajectories, denoted by  $(Z^N_\cdot,\bar Z^N_\cdot)$
(the notation is the same than for processes, which are probabilities on trajectories, but we will always explicitly mention which case we are referring to). Since $\partial_v \varphi$ and $W'$ are bounded, we may write
\begin{align*}
  \Bigl|  \int_s^t  & W'(X^N_u -\bar X^N_u) \chi_\ep(X^N_u-\bar X^N_u)  (1- \xi_\ep(V^N_u-\bar V^N_u) )  \partial_v \varphi(u,Z^N_u) \bigr] \,du \Bigr|  \le  C \, \bigl| A(Z^N,\bar Z^N) \bigr| \\
 & \text{where }\;   A(Z^N,\bar Z^N) :=  \bigl\{  u \in [s,t] \; \text{ s.t. } \;  |X^N_u-\bar X^N_u| < 2 \ep \; \text{ and } \; 
  |V^N_u-\bar V^N_u| >  \ep^\beta   
   \bigr\}.
\end{align*}
We used the notation $|A|$ to denotes the Lebesgue measure of any measurable subset $A$ of $\R$. 
Now, we pick up a time $r \in A(Z^N,\bar Z^N) $. Since $V^N_\cdot$ and $\bar V^N_\cdot$ are almost surely Lipschitz with constant $\frac12$, we may write that 
$$
\forall u \in I^N_r :=  \biggl[ r, r +   \frac{|V^N_r - \bar V^N_r|}2 \biggr], \quad  |V^N_u - \bar V^N_u| \in \biggl[ \frac{|V^N_r - \bar V^N_r|}2,  \frac{3\,|V^N_r - \bar V^N_r|}2 \biggr],
$$
and in particular, the sign of the relative speed do not change on the interval $I^N_r$. 
The condition $\beta< \frac12$ ensures that the two particles will escape the zone $\{ |X^N_u-\bar X^N_u| < 2 \ep \}$ during the time interval $I^N_r$. Precisely, they are inside this zone at time $r$, and will after that have an increase (or decrease) of relative position larger than  $C |V^N_r - \bar V^N_r|^2 \ge C \ep^{2 \beta} > 4 \ep$ if $\ep $ is small enough.

Next, the previous bound by above on the relative velocities, implies that on the time interval $I_r^N$ the two particles can not stay $2\ep$-close too long. 
Such an encounter has a maximal duration of $\frac{8\ep}{|V^N_r - \bar V^N_r|}$. Of course, by periodicity the same particles  may have several close encounters on $I^N_r$ if their relative speed is large, but they can not have more than $C |V^N_r - \bar V^N_r|^2$.
Summing up, we find than the total duration of this close encounters on the time interval $I^N_r$ is smaller than
$ C \ep   |V^N_r - \bar V^N_r| =  C' \ep \bigl| I^N_r \bigr|. $
Covering $A(Z^N_\cdot,\bar Z^N_\cdot)$ with such tile intervals, we end up with
$$
\bigl| A(Z^N,\bar Z^N) \bigr| \le C \ep |s-t|.
$$
Taking now the expectation, we get
\begin{align*}
\E_{(Z^N_\cdot,\bar Z^N_\cdot)} \biggl[  \int_s^t W'(X^N_u -\bar X^N_u) \chi_\ep(X^N_u-\bar X^N_u) 
(1- \xi_\ep(V^N_u-\bar V^N_u) )  \partial_v \varphi(u,Z^N_u) \bigr] \,du  \biggr]  \le C \ep \, |t-s|,
\end{align*}
uniformly in $N$. Remark that since the previous argument only requires that the velocities are almost surely Lipschitz (in time). So
 it also apply with $Z_\cdot$, instead of $Z^N_\cdot$ (See the last remark at the end of Step 2).

\medskip
{\sl Step 5. The effect of close collision with small relative speed.}

For the remaining term, we shall use a symmetry argument similar to the one used by Delort in its proof of  existence of weak  solutions to the 2D Euler equation \cite{Delo91}. 
In fact  using the oddness of the interaction force $W'$
\begin{align*} 
 & \E_{(Z^N_\cdot,\bar Z^N_\cdot)}  \biggl[  \int_s^t W'(X^N_u   -\bar X^N_u)  \chi_\ep(X^N_u-\bar X^N_u) 
\xi_\ep(V^N_u-\bar V^N_u) \partial_v \varphi(u,Z^N_u) \bigr] \,du  \biggr] \\
& =  
\E_{(Z^N_\cdot,\bar Z^N_\cdot)} \biggl[  \int_s^t W'(X^N_u   -\bar X^N_u)  \chi_\ep(X^N_u-\bar X^N_u) 
\xi_\ep(V^N_u-\bar V^N_u) ( \partial_v \varphi(u,Z^N_u) - \partial_v \varphi(u,\bar Z^N_u) \bigr] \,du  \biggr].
\end{align*}
But thanks to the two cut-off functions $\chi$ and $\xi$, we now have for all $(x,v,\bar x, \bar v)$
$$
\chi_\ep(x- \bar x ) 
\xi_\ep(v -\bar v )  \bigl|  \partial_v \varphi(u,x,v) - 
\partial_v \varphi(u,\bar x , \bar v) \bigr|
\le  C  \| \nabla^2 \varphi \|_\infty \ep^\beta,
$$
and this implies that our last expectation is bounded by $C \ep^\beta$, with a constant $C$ 
independent of $N$. That bound also apply when we replace $Z^N$ by $Z$. 


\subsection{Proof of Theorem~\ref{thm:proc}}

The proof of our central result is rather short.

\medskip
Choose an optimal coupling of the random variables $Z^1_0$ and $Z^2_0$. We then couple the processes $Z^1_\cdot$ and $Z_\cdot^2$, using the previous coupling on the initial position independently of what happens after time $0$.  
This is relatively simple if there exists a unique trajectory starting from any initial condition.
In the general case, the coupling is defined has the law on the couples $(Z^1_\cdot, Z^2_\cdot)$ of trajectories such that for any smooth function $\varphi$ on $(\T\times \R)^2$, and  $\psi^1$ and $\psi^2$ on $C[0,\infty,\T \times \R)$
\begin{align*}
\E_{(Z^1_\cdot, Z^2_\cdot)} \bigl[ \varphi(Z^1_0,Z^2_0) \psi^1(Z^1_\cdot) \psi^2(Z^2_\cdot )\bigr] = 
\E_{(Z^1_0, Z^2_0)} \bigl[\varphi(Z^1_0,Z^2_0) \bigr]
\times 
\E_{Z^1_\cdot} \bigl[\psi^1(Z^1_\cdot) \big| Z^1_0 \bigr]
\times 
\E_{Z^2_\cdot} \bigl[ \psi^2(Z^2_\cdot) \big| Z^2_0 \bigr].
\end{align*}
 We also denote by $(\bar Z^1_\cdot,\bar Z^2_\cdot)$ 
 an  independent copy of the couple $(Z^1_\cdot,Z^2_\cdot)$. We can now
differentiate $\E[|Z^1_t - Z_t^2|]$ with respect to time and get
\begin{align*}
\frac d{dt} \E\bigl[|Z^1_t - Z_t^2| \bigr] & \le 
\E\bigl[|V^1_t - V_t^2|\bigr]
+ \E\bigl[|W'(X^1_t - \bar X_t^1) - W'(X^2_t - \bar X_t^2)|\bigr] \\
& \le \sqrt 2 \, \E\bigl[|Z^1_t - Z_t^2| \bigr]
+ \E \bigl[|\sign(X^1_t - \bar X_t^1) - \sign(X^2_t - \bar X_t^2)|\bigr]
\end{align*}
To obtain the second line, we use the decomposition~\eqref{eq:Wdec} of $W'$, 
and the fact that \mbox{$a+ b \le \sqrt{2(a^2+b^2)}$} for any $a,b \ge 0$.

To understand the remaining term, we remark first that $X^1_t-X^2_t$ and $\bar 
X^1_t - \bar X^2_t$ are ``small'' since the trajectories of respectively 
$(Z^1_\cdot,Z^2_\cdot)$  and $(\bar Z^1_\cdot, \bar Z^2_\cdot)$ are coupled in 
a more or less optimized way, while $X^1_t - \bar X_t^1$ is usually ``large'' 
since the trajectories of $Z^1_\cdot$ and $\bar Z^1_\cdot$ are independent (the same is true for $Z^2_\cdot$ and $\bar Z^2_\cdot$).

Neglecting the problem raised by the periodization of the $\sign$ function, 
which is not relevant, we see that they $X^1_t - \bar X_t^1$ and $X^2_t - \bar X_t^2$ may have different signs only if
$$
|X^1_t - \bar X_t^1| \le |X^1_t-X^2_t| + |\bar X^1_t - \bar X^2_t|
\le 2 \max\bigl(|X^1_t-X^2_t|,|\bar X^1_t - \bar X^2_t|\bigr) .
$$
This condition is not symmetric with respect to $Z^1_\cdot$ and $Z^2_\cdot$, 
but this asymmetry is very important in order to obtain  our weak-strong stability 
principle. Next
\begin{align*}
\E\bigl[|\sign(X^1_t - \bar X_t^1) - \sign(X^2_t - \bar X_t^2)| \bigr]
& \le \E\bigl[ \indic_{|
X^1_t - \bar X_t^1| \le 2 \max\bigl(|X^1_t-X^2_t|,|\bar X^1_t - \bar 
X^2_t|\bigr)} \bigr] \\
& \le 
\E\bigl[ \indic_{|
X^1_t - \bar X_t^1| \le 2 |X^1_t-X^2_t|} \bigr]
+
\E\bigl[ \indic_{|
X^1_t - \bar X_t^1| \le 2 |\bar X^1_t- \bar X^2_t|} \bigr] \\
& \le 2\, \E\bigl[ \indic_{|
X^1_t - \bar X_t^1| \le 2 |X^1_t-X^2_t|} \bigr] \\
& \le 2 \, \E_{Z^1_\cdot,Z^2_\cdot} \Bigl[  \E_{\bar Z^1_\cdot} 
\bigl[  \indic_{|
X^1_t - \bar X_t^1| \le 2 |X^1_t-X^2_t|}  \bigr]
\Bigr] \\
& \le  8 \| \rho_t\|_\infty \E \bigl[ |X^1_t-X^2_t| \bigr].
\end{align*}
In the third line, we have used the symmetry in $Z_\cdot, \bar Z_\cdot$. 
The last line is obtained by using the independence of $\bar Z^1$ with respect 
to $(Z^1_\cdot,Z^2_\cdot)$ and the assumption that $\bar X^1_t$ has a bounded density 
$\rho_t$. All in all, we get that
$$
\frac d{dt} \E \bigl[|Z^1_t - Z_t^2|\bigr] \le (\sqrt 2 + 8 \| \rho_t\|_\infty )\E \bigl[|Z^1_t - 
Z_t^2|\bigr].
$$
This concludes the proof.

\subsection{Proof of Proposition~\ref{prop:exisVP}}

The proof is made in four steps. In the first one, we construct an approximating sequence.
In the second one, we establish an a priori estimate on the density.  In the third, we show that the approximating sequence has the Cauchy property. Finally, in the last one we prove that the limit process is a solution of the nonlinear ODE~\eqref{eq:ODENL}.

\medskip
{\sl Step 1. Construction of an approximating sequence.}
We will as usual mollify the interaction force $W'$. For $\ep \in \bigl( 0, \frac12 \bigr)$ we define a piecewise linear and continuous approximation of $W'$ by
$$
 W_\ep(x) :=  
 \begin{cases}
  \frac{x^2 - |x|}2  &\text{ if } \; |x| \in \bigl(\ep, \frac12\bigr) \\
 - \bigl(\frac1{2\ep} -1 \bigr) \frac{x^2}2 - \frac \ep 4 & \text{ if } \; x \in [-\ep,\ep]
\end{cases}, 
 \qquad 
 - W'_\ep(x) := 
 \begin{cases}
  -\frac12 -x  &\text{ if } \; x \in \bigl[-\frac12, - \ep\bigr) \\
 \frac12 -x &\text{ if } \; x \in \bigl(\ep, \frac12\bigr) \\
 \bigl(\frac1{2\ep} -1 \bigr) x & \text{ if } \; x \in [-\ep,\ep]
\end{cases}.
$$
Remark that  $W'_\ep$ is bounded by $\frac12$ and has Lipschitz regularity with constant $\frac1{2\ep}$, and that $W_\ep$ (as $W$) take values in $\bigl[ -\frac18,0 \bigr]$. We now consider the nonlinear ODE with the interaction potential  $W_\ep$. In others words, we look at solutions, in the sense of Definition~\eqref{def:ODENL}, to
\begin{equation} \label{eq:ODENLep}
\forall  t \in [0,+\infty), \qquad 
  X^\ep_t = X^\ep_0 + \int_0^t V^\ep_s \,ds, \qquad
 V^\ep_t = V^\ep_0  - \int_0^t \E_{\bar Z^\ep_\cdot} \bigl[ W'_\ep(X^\ep_s- \bar X^\ep_s)  \bigr] \,ds,
 \end{equation}
with initial data with law $f_0$. But the well-posedness of the associated ``smoothed" Vlasov-Poisson equation is now well understood. The existence and uniqueness of solutions is known even if $f_0$ is a measure. In that setting, the first important works were done by Braun and Hepp \cite{BraHep77}, and Dobrushin \cite{Dobr79}, and Neunzert and Wick \cite{Neun79}. 
Once we have a solution of the mollified VP equation~\eqref{eq:Vla}, we can use the second part of Lemma~\ref{lem:ODENL_mart} to obtain a process $Z^\ep_\cdot$ solution of the nonlinear ODE~\eqref{eq:ODENLep}. But, a (maybe) better strategy will be to adapt on of the proofs mentioned above to the 
setting of nonlinear ODE. In fact, this will correspond to adapt our proof of the weak-strong stability result to the more simple case of Lipschitz interaction forces.

\medskip
{\sl Step 2. An a priori estimates on $\| \rho^\ep_t \|_\infty$.}
From the equations~\eqref{eq:ODENLep} satisfied by $Z^\ep_\cdot$, we get  for all $t \in \R^+$
$$
| V_t^\ep|  \le |V_0^\ep| + \int^t_0 \E_{\bar Z^\ep_\cdot }\bigl[ | W'(X^\ep_s - \bar X^\ep_s| \bigr] \,ds  
\le  |V_0^\ep| + \frac t 2.
$$
Since $f^\ep$ is constant along the trajectories in our smooth setting, we get using~\eqref{eq:condL1} that
$$
| f^\ep_t(X^\ep_t,V_t^\ep) | =  | f_0(X^\ep_0,V^\ep_0) | \le g_0(|V^\ep_0|)
\le g_0 \Bigl(  |V_t^\ep| - \frac t 2 \Bigr).
$$
We used that $g_0$ is decreasing. Remark that the previous bound has no sense if $|V_0^\ep| \le \frac t 2$.
But in that case, we always have $| f^\ep_t(X^\ep_t,V_t^\ep) | \le \|f_0\|_\infty = g_0(0)$. 
So it will be consistent if we extend $g$ on $\R^-$, setting $g(x) = g(0)$ for all $x<0$.
So for a trajectory such that $(X^\ep_t,V^\ep_t) =(x,v)$, we get with that convention
$$
f_t^\ep(x,v) \le g_0 \Bigl(  |v| - \frac t 2 \Bigr).
$$
Integration on the variable $v$ leads to the bound
\begin{equation} \label{estim:Rho}
 \rho^\ep_t(x) = \int_\R  f_t^\ep(x,v) \,dv \le   \int_\R g_0 \Bigl(  |v| - \frac t 2 \Bigr) \,dv
 \le g(0)\,  t  +  2 \int_0^{+\infty} g_0(v) \,dv.
\end{equation}

{\sl Step 3. The Cauchy property of the sequence $Z^\ep_t$.}

Thanks to the a priori bound~\eqref{estim:Rho}, we may use a variant of the stability argument of Theorem~\ref{thm:proc}. In fact, if we want to compare two processes $Z^\ep_\cdot$ and $Z^{\ep'}_\cdot$ for two different $\ep,\ep'$, we can basically apply the same strategy than in the proof of that theorem with $Z^1_\cdot = Z^\ep_\cdot$ and $Z^2_\cdot=Z^{\ep'}_\cdot$. The only difference is that since $Z^\ep_\cdot$ is driven by the interaction potential $W_\ep$ and $Z^{\ep'}_\cdot$ is driven by $W_{\ep'}$, we will get an extra term (and also a less important  additional $\ep$)
\begin{equation} \label{estim:Zep}
\frac d{dt} \E[|Z^\ep_t - Z_t^{\ep'}|] \le (\sqrt 2 + 8 \| \rho^\ep_t\|_\infty ) \bigr( \E[|Z^\ep_t - 
Z_t^{\ep'}|] + 2 \ep' \bigr) + 
\E \bigl[|W'_\ep(X^\ep_t - \bar X_t^\ep) - W'_{\ep'}(X^\ep_t - \bar X_t^\ep)| \bigr]. 
\end{equation}
But, from the definition of $W'_\ep$ and, we have
$$
|W'_\ep(X^\ep_t - \bar X_t^\ep) - W'_{\ep'}(X^\ep_t - \bar X_t^\ep)| \le \frac12 \indic_{|X^\ep_t - \bar X_t^\ep| \le \max(\ep,\ep')}.
$$
But, if we take the expectation and use the bound on $\rho^\ep_t$ has in the proof of Theorem~\ref{thm:proc}, we can bound the second term in the right hand side of~\eqref{estim:Zep}  by
$$
\E \bigl[|W'_\ep(X^\ep_t - \bar X_t^\ep) - W'_{\ep'}(X^\ep_t - \bar X_t^\ep)| \bigr]
 \le
 \| \rho^\ep_t\|_\infty  \max(\ep,\ep'),
$$
and finally get 
$$
\frac d{dt} \E[|Z^\ep_t - Z_t^{\ep'}|] \le (\sqrt 2 + 9 \| \rho^\ep_t\|_\infty ) \bigr( \E[|Z^\ep_t - 
Z_t^{\ep'}|] + 2 \max(\ep,\ep')  \bigr). 
$$
After slight modifications, this calculation will also lead to the stronger statement that there exists a constant $C_t$ such that
$$
\E\Bigl[ \sup_{s \le t}  |Z^\ep_s - Z^{\ep'}_s|\Bigr] \le  C_t \max(\ep,\ep')
$$
This implies that the $Z^\ep_\cdot$ form a Cauchy sequence for the weak topology on $\mathcal P \bigl(C([0,+\infty),\T \times \R)\bigr)$.  
So in particular they converges towards a stochastic process $Z_\cdot$.  

\medskip
{\sl Step 4. Characterization of the limit process $Z_\cdot$.}
In order to prove that the limit process $Z_\cdot$ is a solution fo the nonlinear ODE~\eqref{eq:ODENL} with the original interaction force $W'$, we must for instance prove that for every time $t >0$,
\begin{equation} \label{estim:Cont}
\E_{Z_\cdot} \Bigl[  \Bigl\vert X_t - X_0 - \int_0^t  V_s \,ds  \Bigr\vert \Bigr]=0,
\qquad
\E_{Z_\cdot} \Bigl[  \Bigl\vert V_t - V_0 + \int_0^t  \E_{\bar Z_\cdot}\bigl[ W'(X_s-\bar X_s)\bigr] \,ds  \Bigr\vert \Bigr]=0.
\end{equation}
The first point is simple. For a fixed time $t$, the application that maps a trajectory $\tilde Z_\cdot \mapsto \Bigl\vert \tilde X_t - \tilde X_0 - \int_0^t  \tilde V_s \,ds  \Bigr\vert$ is continuous on $C([0,+\infty),\T \times \R)$. It implies, see \cite[p.30]{Bill}
that the application that sent a process $Z_\cdot$ to the first expectation in~\eqref{estim:Cont} is continuous. Since the expectation is already equals to zero for the approximated process $Z^\ep_\cdot$, we can pass to the limit and get the required equality.

For the second expectation, we have to be more careful because of the  discontinuities in the interaction force. But if we define $F(x) := - \E_{Z_\cdot}\bigl[ W'(x-\bar X_s)\bigr]$ (where $Z_\cdot$ stand for the fixed limit process), then the bound on the time marginals $\rho_t$ of $X_\cdot$ implies that $F$ is continuous. Then, the application
$$
\tilde Z_\cdot \mapsto \E_{\tilde Z_\cdot} \Bigl[  \Bigl\vert V_t - V_0 - \int_0^t  F(X_s) \,ds  \Bigr\vert \Bigr],
$$
 is continuous. The problem is now that the previous expectation is not $0$ for the approximated process. It will vanish for the approximated process only if we replace $F$ by  $F^\ep(x) = - \E_{\bar Z^\ep_\cdot}\bigl[ W'_\ep(x-\bar X_s)\bigr]$. But even if it is not emphasize, we have proved in the second step that $F^\ep$ is a Cauchy sequence in $L^1(\T)$, and in particular it converges strongly in the $L^1$-norm towards $F$. This and the uniform bound on the density  allows to pass to the limit and we obtain  that the second expectation in~\eqref{estim:Cont} also vanishes. This concludes the proof.

\newcommand{\etalchar}[1]{$^{#1}$}
\def\cprime{$'$} \def\cprime{$'$} \def\cprime{$'$}

\end{document}